\RequirePackage[l2tabu,orthodox]{nag}

\documentclass[11pt,a4paper]{amsart}

\usepackage[T1]{fontenc}
\usepackage{lmodern}
\usepackage[utf8]{inputenc}
\usepackage[english]{babel}
\usepackage{amsmath,amsfonts,amssymb,mathrsfs}
\usepackage{graphicx}
\usepackage{subcaption}
\usepackage{tikz}
\usepackage[all]{xy}
\usepackage{xcolor}
\usepackage[expansion=false]{microtype}
\usepackage{hyperref}

\tikzstyle{vertex}=[circle,fill=white,draw,inner sep=0pt,minimum size=5pt]
\tikzstyle{vortex}=[circle,fill=lightgray,draw,inner sep=0pt,minimum size=5pt]
\tikzstyle{vartex}=[circle,fill=black,draw,inner sep=0pt,minimum size=5pt]
\newcommand{\vertex}{\node[vertex]}

\tikzstyle{bvertex}=[circle,fill=white,draw,inner sep=0pt,minimum size=3pt]
\tikzstyle{bvortex}=[circle,fill=lightgray,draw,inner sep=0pt,minimum size=3pt]
\tikzstyle{bvartex}=[circle,fill=black,draw,inner sep=0pt,minimum size=3pt]

\numberwithin{equation}{section}
\numberwithin{figure}{section}

\theoremstyle{plain}
\newtheorem{Thm}{Theorem}[section]

\newtheorem{Prop}[Thm]{Proposition}
\newtheorem{Lemma}[Thm]{Lemma}
\newtheorem{Cor}[Thm]{Corollary}

\theoremstyle{definition}
\newtheorem{Defn}[Thm]{Definition}
\newtheorem{Exam}[Thm]{Example}

\newtheorem*{Def*}{Definition}

\theoremstyle{remark}

\newtheorem*{Remarks}{Remarks}

\let \seq=\subseteq

\let \ToT=\Leftrightarrow
\let \0=\emptyset
\let \1=\infty
\let \al=\alpha
\let \be=\beta

\let \f=\varphi
\let \p=\ldots
\let \h=\text
\let \m=\mathbf
\let \cd=\cdots

\let \t=\theta

\newcommand{\Th}{Theorem}
\newcommand{\Ths}{Theorems}
\newcommand{\Co}{Corollary}

\newcommand{\Pro}{Proposition}

\newcommand{\Def}{Definition}

\newcommand{\Con}{Conjecture}
\newcommand{\Cons}{Conjectures}

\newcommand{\N}{\mathbf{N}}

\newcommand{\id}{\mathrm{id}}

\newcommand{\J}{\mathfrak{I}}
\newcommand{\Br}{\mathrm{Br}}
\newcommand{\s}{\underline{s}}
\newcommand{\z}{\underline{z}}
\newcommand{\Ss}{\underline{S}}
\newcommand{\Rh}{\hat{\mathcal{R}}}
\newcommand{\ord}{\mathrm{ord}}
\newcommand{\ala}{\al_{\s,\s'}}
\newcommand{\alb}{\al_{\s',\s}}
\newcommand{\sss}{\s \, \s}
\newcommand{\ssp}{\s \, \s'}

\newcommand{\move}[1]{\overset{\h{#1}}{\sim}}

\title{A word property for twisted involutions in Coxeter groups}
\author{Mikael Hansson \and Axel Hultman}
\thanks{}
\address{Department of Mathematics, Link\"oping University, SE-581 83 Link\"oping, Sweden}
\email{mikael.hansson@liu.se, axel.hultman@liu.se}

\begin{document}

\begin{abstract}
Given an involutive automorphism $\theta$ of a Coxeter system $(W,S)$, let $\mathfrak{I}(\theta) \subseteq W$ denote the set of twisted involutions. We provide a minimal set of moves that can be added to the braid moves, in order to connect all reduced $\underline{S}$-expressions (also known as admissible sequences, reduced $I_\theta$-expressions, or involution words) for any given $w \in \mathfrak{I}(\theta)$. This can be viewed as an analogue of the well-known word property for Coxeter groups. It improves upon a result of Hamaker, Marberg, and Pawlowski, and generalises similar statements valid in certain types due to Hu, Zhang, Wu, and Marberg.
\end{abstract}

\maketitle

\section{Introduction} \label{intro}

A fundamental result in the theory of Coxeter groups is the \emph{word property}, due to Matsumoto~\cite{Matsumoto} and Tits~\cite{Tits}. Given a Coxeter system $(W,S)$ and two generators $s,s' \in S$ such that $ss'$ has finite order $m(s,s')$, let $\al_{s,s'}$ denote the word $ss's \cd$ of length $m(s,s')$. Operating on words in the free monoid $S^*$, call the replacement of $\al_{s,s'}$ by $\al_{s',s}$ a \emph{braid move}. A word in $S^*$ representing $w \in W$ is called a \emph{reduced word} for $w$ if $w$ cannot be represented by a product of fewer elements of $S$.

\begin{Thm}[Word property \mbox{\cite{Matsumoto,Tits}}] \label{word}
Let $(W,S)$ be a Coxeter system and let $w \in W$. Then any two reduced words for $w$ can be connected by a sequence of braid moves.
\end{Thm}

Given an involutive automorphism $\t$ of $(W,S)$, let
\[
\J(\t)=\{w \in W \mid \t(w)=w^{-1}\}
\]
be the set of \emph{twisted involutions}. Note that $\J(\id)$ is the set of ordinary involutions in $W$. Like $W$, $\J(\t)$ can be described in terms of ``words'' and ``reduced expressions''. Namely, with the alphabet $\Ss=\{\s \mid s \in S\}$, every word in $\Ss^*$ represents a twisted involution, and conversely, every twisted involution can be represented in this way; see Section~\ref{notation} for the details. The shortest possible representatives of $w \in \J(\t)$ are called \emph{reduced $\Ss$-expressions} for $w$; let $\Rh_\t(w)$ be the set of all such representatives.

In combinatorial approaches to Coxeter group theory, the study of (reduced) words is central. Analogously, (reduced) $\Ss$-expressions form the combinatorial foundation for $\J(\t)$; systematic study of $\J(\t)$ was initiated by Richardson and Springer~\cite{R-S1,R-S2,Springer} because of its connections with Borel orbit decompositions of symmetric varieties. In the former case, \Th~\ref{word} is a cornerstone. The goal of this paper is to identify the correct analogue in the latter setting.

In the context of $\Ss$-expressions, a \emph{braid move} is the replacement of $\ala$ by $\alb$. It is known that braid moves preserve $\Rh_\t(w)$ (it follows from \cite{R-S1} when $W$ is a Weyl group, and from \cite{H-M-P} for general $W$), but they do not in general suffice to connect all of $\Rh_\t(w)$. For example, when $W$ is the symmetric group $S_4$, $s_i=(i,i+1)$, and $\t=\id$, the reduced $\Ss$-expressions $\s_2\s_3\s_1\s_2$ and $\s_3\s_2\s_1\s_2$ both represent the same (ordinary) involution, namely the longest group element, but they cannot be connected using only braid moves.

For the purposes of the introduction, a \emph{half-braid move} amounts to replacing $\ssp$ with $\s'\s$ (or $\ssp\s$ with $\s'\ssp$) if these are the first letters in a reduced $\Ss$-expression, and $m(s,s')=3$ ($m(s,s')=4$). A more general definition is given in Section~\ref{applications}.

When $W$ is of type~$A$ and $\t=\id$, Hu and Zhang~\cite[\Th~3.1]{H-Z} proved that braid moves and half-braid moves suffice to connect $\Rh_\t(w)$. They extended this result to types~$B$ and $D$ (also with $\t=\id$) in \cite{H-Z_BD} (\Ths~3.10 and 4.8); in either of these types, braid moves, half-braid moves, and one special, additional move (involving more than two different letters) are enough. With Wu~\cite{H-Z-W}, they presented a similar assertion for type~$F_4$.

Recently, Marberg~\cite[\Th~1.4]{Marberg2} generalised these results to all finite and affine Coxeter groups, for arbitrary $\t$. Aided by a computer, he identified moves that are necessary and sufficient in such groups. He also presented several conjectures about the general situation.

Before \cite{Marberg2}, Hamaker, Marberg, and Pawlowski~\cite[\Th~7.9]{H-M-P} provided an answer for general $W$ and $\t$. Under certain circumstances, they allow replacing $\ssp\s \cd$ with $\s'\ssp \cd$ even if the number of letters is less than $m(s,s')$. The set of all such \emph{involution braid relations} is always sufficient for connecting $\Rh_\t(w)$.

As noted in \cite{H-M-P}, in type~$A$, \cite[\Th~7.9]{H-M-P} requires the addition of many more moves to the braid moves than just the half-braid moves, so it reduces to a weaker statement than \cite[\Th~3.1]{H-Z}. The same is true in types~$B$ and $D$: the set of all involution braid relations is much larger than necessary for connecting $\Rh_\t(w)$. Our main result provides, for any Coxeter system $(W,S)$ with any involutive automorphism $\t$, a minimal set of moves that must be added to the braid moves in order to connect $\Rh_\t(w)$. In fact, we show that Marberg's moves from \cite{Marberg2} suffice in arbitrary Coxeter systems with arbitrary $\t$. In particular, we confirm his aforementioned conjectures. Our approach is different from Marberg's, and thus provides an independent road to his results as a special case.

An \emph{initial move} is the replacement of one element in $\Rh_\t(v)$ (for some $v \in \J(\t)$) by another, in the beginning of a reduced $\Ss$-expression for some $w \in \J(\t)$. Both half-braid moves and the two special Hu-Zhang moves in types~$B$ and $D$ are initial moves. In the former case, $v$ is the longest element of a dihedral parabolic subgroup ($I_2(3)$ in types~$A$ and $D$, and $I_2(3)$ or $I_2(4)$ in type~$B$), and in the latter case, $v$ is the longest element of a parabolic subgroup of type~$B_3$ or $D_4$, respectively.

The following is the main result of the present paper. It describes precisely the initial moves that are necessary and sufficient to add to the braid moves, in order to connect all of $\Rh_\t(w)$. As such, it can be viewed as an analogue of \Th~\ref{word} for $\J(\t)$.

If $J \seq S$ generates a finite parabolic subgroup $W_J$ of $W$, let $w_0(J)$ denote its longest element. When $\t(J)=J$, $\t_J$ is the restriction of $\t$ to $W_J$.

\begin{Thm} \label{MT}
Let $(W,S)$ be a Coxeter system with an involutive automorphism $\t$, and let $w \in \J(\t)$. Then any two reduced $\Ss$-expressions for $w$ can be connected by a sequence of braid moves and initial moves that replace $x$ with $y$ if $x,y \in \Rh_\t(w_0(J))$ for some $\t$-stable $J \seq S$. The following $W_J$ and $\t_J$ are necessary and sufficient:
\begin{itemize}
  \item $W_J$ of type~$A_3$ with $\t_J \neq \id$;
  \item $W_J$ of type~$B_3$;
  \item $W_J$ of type~$D_4$ with $\t_J=\id$;
  \item $W_J$ of type~$H_3$;
  \item $W_J$ of type~$I_2(m)$, $3 \leq m<\1$, with $\t_J=\id$;
  \item $W_J$ of type~$I_2(m)$, $2 \leq m<\1$, with $\t_J \neq \id$.
\end{itemize}
\end{Thm}

\begin{Remarks}
It should be observed that the listed $W_J$ (apart from $I_2(2)$) are precisely the finite types for which the complement of the Coxeter graph is disconnected. This fact is explained in the proof of Lemma~\ref{disconnected}.

Note that there is no need to specify $\t_J$ in types~$B_3$ and $H_3$ since they admit no non-trivial Coxeter system automorphism. Also, observe that the reducible dihedral group $I_2(2)$ does not appear in the list when $\t_J=\id$. The corresponding move \emph{is} allowed, but it is a braid move in this case.
\end{Remarks}

We conclude this section with an outline of the paper. In Section~\ref{notation}, the necessary definitions and previous results are recalled. Then, in Section~\ref{proof}, we prove \Th~\ref{MT}. It follows from the proof that, in fact, only one move of each type is needed, if $x$ and $y$ are appropriately chosen. An explanation is given in Section~\ref{disc}, where a list of such $x$ and $y$ is also provided. This clarifies how the results from \cite{H-Z,H-Z_BD,H-Z-W,Marberg2} are recovered from ours. Finally, in Section~\ref{applications}, consequences of \Th~\ref{MT} in some special cases are discussed.

\section{Notation and preliminaries} \label{notation}

As general references on Coxeter group theory, the reader could consult \cite{B-B} or \cite{Humphreys}. We assume familiarity with the basics but reiterate some of it here in order to agree on notation. Some useful tools for dealing with twisted involutions are also reviewed. For finite Coxeter groups, they appear in Richardson and Springer~\cite{R-S1}. In the general case, everything can be found in \cite{Hultman2,Hultman3} from which our notation is taken.

Let $(W,S)$ be a Coxeter system. In the sequel, we shall often find it important to distinguish notationally between words in the monoid $S^*$ and the elements of $W$ they represent. From now on, a sequence of generators inside square brackets indicates an element of $S^*$, whereas a sequence without brackets is an element of $W$. Thus, if $s_i \in S$, then $[s_1 \cd s_k]\in S^*$ and $s_1 \cd s_k\in W$.

For $w \in W$, its \emph{length} $\ell(w)$ is the smallest integer $k$ such that $w=s_1 \cd s_k$ for some $s_i \in S$; $[s_1 \cd s_k]$ is then called a \emph{reduced word} (or a \emph{reduced expression}) for $w$.

Let $\t$ be an involutive automorphism of $(W,S)$. Recall the set of twisted involutions $\J(\t)$ defined in the introduction. Like $W$, $\J(\t)$ can be described using words, this time in the monoid $\Ss^*$.

\begin{Defn} \label{action}
The free monoid $\Ss^*$ acts from the right on the set $W$ by
\[
w\s=\begin{cases}ws & \h{if $\t(s)ws=w$,} \\ \t(s)ws & \h{otherwise,}\end{cases}
\]
and $w\s_1 \cd \s_k=(\cd((w\s_1)\s_2)\cd)\s_k$. Observe that $w\sss=w$ for all $w \in W$ and all $s \in S$. We write $\s_1 \cd \s_k$ for $e\s_1 \cd \s_k$, where $e \in W$ is the identity element.
\end{Defn}

The orbit of $e$ under this action is precisely $\J(\t)$:

\begin{Lemma}[\mbox{\cite[\Pro~3.5]{Hultman2}}] \label{orbit}
We have
\[
\J(\t)=\{\s_1 \cd \s_k \mid \s_1,\p,\s_k \in \Ss, \: k \in \N\}.
\]
\end{Lemma}

Again, we shall henceforth use square brackets to indicate that a sequence of letters should be interpreted as an element of the monoid. If $w=\s_1 \cd \s_k$ for some $\s_i \in \Ss$, we call $[\s_1 \cd \s_k] \in \Ss^*$ an \emph{$\Ss$-expression} for $w$. It is \emph{reduced} if no $\Ss$-expression for $w$ consists of fewer than $k$ elements of $\Ss$. In this case, $\rho(w)=k$ is the \emph{rank} of $w$.

Given an $\Ss$-expression $\epsilon \in \Ss^*$, $\ord(\epsilon)$ denotes the element of $S^*$ obtained by expanding $\epsilon$ according to \Def~\ref{action}.

\begin{Exam}
Let $W=S_4$ with $s_i=(i,i+1)$. If $\t=\id$, $\ord([\s_1\s_2\s_3\s_2])=[s_3s_2s_1s_2s_3s_2]$, which is a reduced expression for the longest element $w_0 \in W$. On the other hand, if $\t \neq \id$ (i.e., $\t(s_i)=s_{5-i}$), we have $\ord([\s_1\s_2\s_3\s_2])=[s_2s_2s_3s_1s_2s_3s_2]$, which is not reduced.
\end{Exam}

Some remarks about the terminology are in order. We follow the notation used in \cite{Hultman2,Hultman3}. What we call an $\Ss$-expression is the right handed version of an ``$I_*$-expression'' \cite{H-Z,H-Z_BD,H-Z-W,Marberg1}. \emph{Reduced} $\Ss$-expressions are the same as (again, right handed versions of) ``admissible sequences'' \cite{R-S1} and ``involution words'' \cite{H-M-P,Marberg2}. One should furthermore note that Richardson and Springer mostly use a monoid action, denoted by $*$ in \cite{R-S1}, which is different from that defined in \Def~\ref{action} above although they coincide on reduced $\Ss$-expressions (i.e., admissible sequences).

\begin{Defn} \label{absolute}
Let $[\s_1 \cd \s_k]$ be a reduced $\Ss$-expression for $w \in \J(\t)$. Then the \emph{twisted absolute length} of $w$, denoted $\ell^\t(w)$, is the number of indices $i \in [k]$ such that $\s_1 \cd \s_{i-1}\s_i=\s_1 \cd \s_{i-1}s_i$.
\end{Defn}

It follows from \cite[\Pro~2.5]{Hultman3} that this definition of $\ell^\t(w)$ is independent of the choice of reduced $\Ss$-expression for $w$. When $\t=\id$, $\ell^\t(w)$ coincides with the \emph{absolute length} of $w$, i.e., the smallest number of reflections whose product is $w$, see \cite{Hultman1}.

Given $w \in W$, let $D_R(w)=\{s \in S \mid \ell(ws)<\ell(w)\}$ and $D_L(w)=\{s \in S \mid \ell(sw)<\ell(w)\}$ be the sets of \emph{right descents} and \emph{left descents}, respectively. Observe that $D_R(w)=D_L(\t(w))$ if $w \in \J(\t)$.

\begin{Lemma}[\mbox{\cite{Hultman2}}] \label{mixed}
If $w \in \J(\t)$ and $s \in S$, then $\rho(w\s)=\rho(w) \pm 1$. Moreover, the following statements are equivalent:
\begin{itemize}
  \item $s \in D_R(w)$;
  \item $\rho(w\s)=\rho(w)-1$;
  \item some reduced $\Ss$-expression for $w$ ends with $\s$.
\end{itemize}
\end{Lemma}

The \emph{exchange property} is a fundamental property of Coxeter groups; in fact, it characterises them among groups generated by involutions, see \cite{Matsumoto}.

\begin{Lemma}[Exchange property]
If $[s_1 \cd s_k]$ is a reduced expression and $s \in D_R(s_1 \cd s_k)$, then $s_1 \cd s_ks=s_1 \cd s_{i-1}s_{i+1} \cd s_k$ for some $i \in [k]$.
\end{Lemma}

An analogous result holds for twisted involutions:

\begin{Lemma}[\mbox{\cite[\Pro~3.10]{Hultman2}}] \label{exchange}
If $[\s_1 \cd \s_k]$ is a reduced $\Ss$-expression and $s \in D_R(\s_1 \cd \s_k)$, then $\s_1 \cd \s_k\s=\s_1 \cd \s_{i-1}\s_{i+1} \cd \s_k$ for some $i \in [k]$.
\end{Lemma}

The \emph{subword property}, due to Chevalley~\cite{Chevalley}, characterises the \emph{Bruhat order} on $W$.

\begin{Lemma}[Subword property] \label{subword1}
Let $u,w \in W$ and suppose $[s_1 \cd s_k]$ is a reduced expression for $w$. Then $u \leq w$ in the Bruhat order if and only if $[s_{i_1} \cd s_{i_m}]$ is a reduced expression for $u$, for some $1 \leq i_1<\cd<i_m \leq k$.
\end{Lemma}

Again there is an analogous result for twisted involutions:

\begin{Lemma}[\mbox{\cite[\Th~2.8]{Hultman3}}] \label{subword2}
Let $u,w \in \J(\t)$ and suppose $[\s_1 \cd \s_k]$ is a reduced $\Ss$-expression for $w$. Then $u \leq w$ in the Bruhat order if and only if $[\s_{i_1} \cd \s_{i_m}]$ is a reduced $\Ss$-expression for $u$, for some $1 \leq i_1<\cd<i_m \leq k$.
\end{Lemma}

\section{A word property for reduced \texorpdfstring{$\Ss$}{}-expressions} \label{proof}

\subsection{Lemmas} \label{lemmas}

In order to establish \Th~\ref{MT}, it is crucial to understand when $w \in \J(\t)$ admits reduced $\Ss$-expressions that end with the longest possible alternating sequence of two given right descents.

\begin{Defn} \label{maximal}
Given $w \in \J(\t)$ and $s,s' \in D_R(w)$, $s \neq s'$, say that $w$ is \emph{$(s,s')$-maximal} if it has a reduced $\Ss$-expression of the form
\[
[\s_1 \cd \s_k\underbrace{\cd \s'\ssp}_{m(s,s')}] \in \Rh_\t(w).
\]
\end{Defn}

Since braid moves in a reduced $\Ss$-expression preserve the twisted involution, any $(s,s')$-maximal element is also $(s',s)$-maximal. A more general assertion is provided by Lemma~\ref{group} below.

The main goal of this subsection is to provide a characterisation of the $(s,s')$-maximal elements $w \in \J(\t)$; this is Lemma~\ref{ML} below. A closely related description, in terms of the minimal coset representative of $w$ in $W/W_{\{s,s'\}}$, can be gleaned from the proof of \cite[\Th~7.9]{H-M-P}.\footnote{Namely, with the notation of \cite{H-M-P}, said proof shows that the number $m(s,t,\t_{\m{a}})$ is the length of the longest sequence $s,t,s,\p$ which can be attached to $\m{a}$, if $s,t \notin D_R(\hat{\delta}_*(\m{a}))$, without destroying reducedness.} In particular, the following lemma is a consequence of that description. We provide a short independent argument.

\begin{Lemma} \label{group}
Suppose $w \in \J(\t)$ and $s,s' \in D_R(w)$. If
\[
\epsilon=[\s_1 \cd \s_k\underbrace{\cd \s'\ssp}_{\h{$\al$ letters}}] \in \Rh_\t(w)
\]
with $\al$ maximal, then
\[
\epsilon'=[\s_1 \cd \s_k\underbrace{\cd \ssp\s}_{\h{$\al$ letters}}] \in \Rh_\t(w).
\]
\end{Lemma}

\begin{proof}
By Lemma~\ref{exchange}, a reduced $\Ss$-expression $\epsilon''$ for $w=w\sss$ is obtained by deleting a letter from $\epsilon$ and appending $\s$ at the end. Maximality of $\al$ and rank considerations show that the deleted letter necessarily is the one immediately to the right of $\s_k$. Hence, $\epsilon''=\epsilon'$.
\end{proof}

\begin{Lemma} \label{reduced}
Let $\epsilon$ be a reduced $\Ss$-expression. Then $\ord(\epsilon)$ is also reduced. If, moreover,
\[
\epsilon=[\s_1 \cd \s_k\underbrace{\cd \s'\ssp}_{\h{$\al$ letters}}],
\]
then
\[
\ord(\epsilon)=[\underbrace{\cd \t(s')\t(s)\t(s') \cd}_{\h{$\be$ letters}}\ord(\s_1 \cd \s_k)\underbrace{\cd s'ss'}_{\h{$\al$ letters}}]
\]
for some $\al-2 \leq \be \leq \al$.
\end{Lemma}

\begin{proof}
By \cite[\Th~4.8]{Hultman1}, $\ell=2\rho-\ell^\t$. The first claim thus follows from \Def~\ref{absolute}. For the second, observe that for any $u \in \J(\t)$ and $s \in S$, $\ell^\t(u\s)>\ell^\t(u)$ implies $D_R(u\s) \supset D_R(u)$. Since $\epsilon$ is reduced, this means that at most two of its rightmost $\al$ letters (namely, the first and the last) can contribute to $\ell^\t(\s_1 \cd \s_k \cd \s'\ssp)$. This yields the second claim.
\end{proof}

\begin{Lemma} \label{reversed}
If $ws=\t(s')w$ for some $w \in \J(\t)$ and $s,s' \in S$, then $ws'=\t(s)w$.
\end{Lemma}

\begin{proof}
We have $ws'=(s'\t(w))^{-1}=(\t(\t(s')w))^{-1}=(\t(ws))^{-1}=\t(s)w$.
\end{proof}

\begin{Lemma} \label{symm}
Let $w \in \J(\t)$ and $s,s' \in D_R(w)$, and suppose $w$ is not $(s,s')$-maximal. Then $ws=\t(s)w$ if and only if $ws'=\t(s')w$.
\end{Lemma}

\begin{proof}
By Lemmas~\ref{group} and \ref{reduced}, $w$ has reduced $\Ss$-expressions
\[
\epsilon=[\s_1 \cd \s_k\underbrace{\cd \s'\ssp}_\al] \quad \h{and} \quad \epsilon'=[\s_1 \cd \s_k\underbrace{\cd \ssp\s}_\al]
\]
with corresponding reduced words
\[
\ord(\epsilon)=[\underbrace{\cd \t(s')\t(s)\t(s') \cd}_\be\ord(\s_1 \cd \s_k)\underbrace{\cd s'ss'}_\al]
\]
and
\[
\ord(\epsilon')=[\underbrace{\cd \t(s)\t(s')\t(s) \cd}_\be\ord(\s_1 \cd \s_k)\underbrace{\cd ss's}_\al]
\]
for some $\be \leq \al$. Now, $ws=\t(s)w \ToT w\s=ws \ToT w=(w\s)s$, which is the case if and only if $\ord(\epsilon')$ does not begin with $\t(s)$. Similarly, $ws'=\t(s')w$ precisely when $\ord(\epsilon)$ does not begin with $\t(s')$. Since $\al<m(s,s')$, and $\ord(\epsilon)$ and $\ord(\epsilon')$ represent the same element $w$, they begin with different letters. Hence, it cannot be that exactly one of $ws=\t(s)w$ and $ws'=\t(s')w$ holds.
\end{proof}

The promised characterisation of $(s,s')$-maximal elements can now be delivered. It is the main technical ingredient in the proofs of Lemmas~\ref{parabolic} and \ref{disconnected}, which are needed in order to establish \Th~\ref{MT}.

\begin{Lemma} \label{ML}
Given $w \in \J(\t)$, let $s,s' \in D_R(w)$ with $s \neq s'$. Then $w$ is $(s,s')$-maximal if and only if either
\begin{itemize}
  \item $m(s,s')=2$ and $ws \neq \t(s')w$
\end{itemize}
or
\begin{itemize}
  \item $m(s,s') \geq 3$ and $\{ws,ws'\} \neq \{\t(s)w,\t(s')w\}$.
\end{itemize}
\end{Lemma}

\begin{proof}
Choose a reduced $\Ss$-expression for $w$,
\[
\epsilon=[\s_1 \cd \s_k\underbrace{\cd \s'\ssp}_\al] \in \Rh_\t(w),
\]
with $\al$ maximal.

The proof is divided into three parts. In the first part, we show that if $\{ws,ws'\} \neq \{\t(s)w,\t(s')w\}$, then $\al=m(s,s')$. Then, in the second part, we show that if $ws=\t(s')w$ and $ws'=\t(s)w$, then $\al<m(s,s')$. Finally, we show that if $ws=\t(s)w$ and $ws'=\t(s')w$, then $\al=m(s,s')$ if $m(s,s')=2$, and $\al<m(s,s')$ if $m(s,s') \geq 3$.

First, assume $\{ws,ws'\} \neq \{\t(s)w,\t(s')w\}$. In order to obtain a contradiction, suppose $\al<m(s,s')$. Then, by Lemmas~\ref{reversed} and \ref{symm}, $\t(s)w \notin \{ws,ws'\}$ and $\t(s')w \notin \{ws,ws'\}$. From Lemma~\ref{reduced}, we know that $\ord(\epsilon)$ is reduced, and since $ws' \neq \t(s')w$, we have
\[
\ord(\epsilon)=[\underbrace{\t(s')\t(s)\t(s') \cd}_\be\ord(\s_1 \cd \s_k)\underbrace{\cd s'ss'}_\al]
\]
for some $1 \leq \be \leq \al$.

Consider the element $ws$. By the exchange property, a reduced word for $ws$ is obtained by deleting a letter in $\ord(\epsilon)$. Since $\al<m(s,s')$, $s$ is not a right descent of $\s_1 \cd \s_k \cd s'ss'$. Hence the deleted letter is one of the leftmost $\be$ generators. It is not the very first one, because $ws \neq \t(s')w$. It cannot be any of the other ones not adjacent to $\ord(\s_1 \cd \s_k)$ since $\ell(ws)=\ell(w)-1$. Hence it is the generator immediately to the left of $\ord(\s_1 \cd \s_k)$. Thus,
\[
w=(ws)s=\underbrace{\t(s')\t(s)\t(s') \cd}_{\be-1}\ord(\s_1 \cd \s_k)\underbrace{\cd s'ss's}_{\al+1}.
\]
Now considering $ws'$, and continuing in this way, we finally obtain
\[
w=\underbrace{\t(s')\t(s)\t(s') \cd}_{\be+\al-m(s,s')}\ord(\s_1 \cd \s_k)\underbrace{\cd s'ss's \cd}_{m(s,s')}.
\]
Observe that $\t(s) \in D_L(w)$ and $\be+\al-m(s,s')<m(s,s')$. By the exchange property and similar reasoning as above, this however means that
\[
\t(s)w=\underbrace{\t(s')\t(s)\t(s') \cd}_{\be+\al-m(s,s')}\ord(\s_1 \cd \s_k)\underbrace{\cd s'ss's \cd}_{m(s,s')-1},
\]
which is equal to either $ws$ or $ws'$, a contradiction. Hence, $\al=m(s,s')$.

Next, assume $ws=\t(s')w$ and $ws'=\t(s)w$. We have
\[
\ord(\epsilon)=[\underbrace{\t(s')\t(s)\t(s') \cd}_\be\ord(\s_1 \cd \s_k)\underbrace{\cd s'ss'}_\al]
\]
with $1 \leq \be \leq \al$. Thus,
\[
w=\t(s')ws=\underbrace{\t(s)\t(s') \cd}_{\be-1}\s_1 \cd \s_k\underbrace{\cd s'ss's}_{\al+1}.
\]
Since $\ord(\epsilon)$ is reduced, $\al<m(s,s')$.

Finally, suppose $ws=\t(s)w$ and $ws'=\t(s')w$. This implies that
\[
\ord(\epsilon)=[\underbrace{\t(s)\t(s')\t(s) \cd}_\be\ord(\s_1 \cd \s_k)\underbrace{\cd s'ss'}_\al]
\]
for some $0 \leq \be<\al$. In case $\be \geq 1$, we obtain
\[
w=\t(s)ws=\underbrace{\t(s')\t(s) \cd}_{\be-1}\s_1 \cd \s_k\underbrace{\cd s'ss's}_{\al+1},
\]
which exactly as above leads to $\al<m(s,s')$, and hence, $m(s,s') \geq 3$. If $\be=0$, then $w=\s_1 \cd \s_k \cd s'ss'=\s_1 \cd \s_k \cd ss's$, whence $m(s,s')=\al$. By Lemma~\ref{reduced}, $\al \leq 2$, so $m(s,s')=\al=2$.
\end{proof}

It is convenient to encode the information conveyed by Lemma~\ref{ML} in a graph. To this end, let $G(w,\t)$ be the graph on vertex set $D_R(w)$ in which $\{s,s'\}$ is an edge if and only if $w$ is $(s,s')$-maximal.

\begin{Exam} \label{exam}
Figure~\ref{graph} displays $G(w,\t)$ when $w$ is the longest element in $S_4$. In this case, $D_R(w)=\{s_1,s_2,s_3\}$ with $s_i=(i,i+1)$. We have $ws_1=s_3w$, $ws_2=s_2w$, and $ws_3=s_1w$. Let us use Lemma~\ref{ML} to determine the edge set. Consider first $\t=\id$. Then $ws_1=\t(s_3)w$, so $w$ is not $(s_1,s_3)$-maximal. On the other hand, $ws_1 \notin \{\t(s_1)w,\t(s_2)w\}$ and $ws_3 \notin \{\t(s_2)w,\t(s_3)w\}$, whence $w$ is $(s_1,s_2)$- and $(s_2,s_3)$-maximal. Now assume $\t \neq \id$. Then $ws_1 \neq \t(s_3)w$, so $w$ is $(s_1,s_3)$-maximal. However, $ws_1=\t(s_1)w$ and $ws_2=\t(s_2)w$, whence $w$ is not $(s_1,s_2)$-maximal. Similarly, $w$ is not $(s_2,s_3)$-maximal either.
\end{Exam}

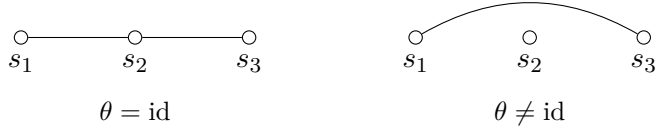
\begin{figure}[tb]
\begin{subfigure}[b]{0.4\textwidth}
\centering
\begin{tikzpicture}[scale=1.5]
  \vertex (0) at (0,0) [label=below:$s_1$] {}; \vertex (1) at (1,0) [label=below:$s_2$] {}; \vertex (2) at (2,0) [label=below:$s_3$] {};
  \draw (0)--(1)--(2);
\end{tikzpicture}
\caption*{$\t=\id$}
\bigskip
\end{subfigure}
\begin{subfigure}[b]{0.4\textwidth}
\centering
\begin{tikzpicture}[scale=1.5]
  \vertex (0) at (0,0) [label=below:$s_1$] {}; \vertex (1) at (1,0) [label=below:$s_2$] {}; \vertex (2) at (2,0) [label=below:$s_3$] {};
  \draw (0) to [out=30,in=150] (2);
\end{tikzpicture}
\caption*{$\t \neq \id$}
\bigskip
\end{subfigure}
\caption{The graph $G(w,\t)$, where $w$ is the longest element in $S_4$. Observe that $G(w,\t)$ is connected when $\t=\id$, but disconnected when $\t \neq \id$.} \label{graph}
\end{figure}

\subsection{Proof of the main result} \label{bevis}

Before commencing to prove \Th~\ref{MT}, let us describe the basic strategy. Starting with a list comprised of all braid moves, we work by induction on the rank of a given twisted involution $w$. If the moves that are so far collected do not suffice to connect all the reduced $\Ss$-expressions for $w$, sufficient new ones are added to the list. As we shall see, new moves must be added precisely when $G(w,\t)$ is disconnected. Moreover, this can only happen when $w$ is the longest element of a finite, $\t$-stable parabolic subgroup $W_J$ (see Lemma~\ref{parabolic}). Using the classification of finite Coxeter groups, we obtain the list of such $W_J$ (see Lemma~\ref{disconnected}).

We now turn to the actual proof. First, the connectedness of $G(w,\t)$ is investigated.

\begin{Lemma} \label{parabolic}
Let $w \in \J(\t)$. If $G(w,\t)$ is disconnected, then $w$ is the longest element $w_0(J)$ of a finite, $\t$-stable parabolic subgroup $W_J$, $J \seq S$.
\end{Lemma}

\begin{proof}
Say that $s \in S$ \emph{passes through $w$} if $ws=\tilde{s}w$ for some $\tilde{s} \in S$. It is an immediate consequence of Lemma~\ref{ML} that $s \in D_R(w)$ is adjacent to all other vertices of $G(w,\t)$ if $s$ does not pass through $w$. Hence, if $G(w,\t)$ is disconnected, then every right descent of $w$ passes through $w$. We claim that in this case, $w$ is the longest element of the (hence finite) parabolic subgroup generated by $S(w)$, where $S(w)$ is the set of generators that appear in some reduced word for $w$; by the word property, it is independent of the choice of reduced expression. Indeed, if $[s_1 \cd s_k]$ is a reduced word for $w$, we have
\[
w=s_1 \cd s_k=\tilde{s}_ks_1 \cd s_{k-1}=\tilde{s}_{k-1}\tilde{s}_ks_1 \cd s_{k-2}=\cd=\tilde{s}_2\tilde{s}_3 \cd \tilde{s}_ks_1,
\]
implying that $S(w)=D_R(w)$, a property which is equivalent to $w=w_0(S(w))$, see \cite[Lemma 7.11]{Davis}.

Clearly, $S(w)$ is $\t$-stable for any $w \in \J(\t)$. Therefore, $G(w,\t)$ is always connected, unless, possibly, $w=w_0(J)$ for some $\t$-stable $J \seq S$.
\end{proof}

\begin{Lemma} \label{disconnected}
If $W$ is finite with longest element $w_0$, then $G(w_0,\t)$ is disconnected in exactly the following cases:
\begin{itemize}
  \item $W$ of type~$A_3$ with $\t \neq \id$;
  \item $W$ of type~$B_3$;
  \item $W$ of type~$D_4$ with $\t=\id$;
  \item $W$ of type~$H_3$;
  \item $W$ of type~$I_2(m)$, $3 \leq m<\1$, with $\t=\id$.
  \item $W$ of type~$I_2(m)$, $2 \leq m<\1$, with $\t \neq \id$.
\end{itemize}
Moreover, when it is disconnected, it has exactly two connected components.
\end{Lemma}

\begin{proof}
Let $\f$ denote the involutive automorphism $x \mapsto w_0xw_0$. It is convenient to reformulate the conditions stated in Lemma~\ref{ML} in terms of $\f$. Namely, notice that $w_0s \neq \t(s')w_0$ is equivalent to $\f(s) \neq \t(s')$ and that $\{w_0s,w_0s'\} \neq \{\t(s)w_0,\t(s')w_0\}$ if and only if $\{ \f(s),\f(s')\} \neq \{\t(s),\t(s')\}$.

Now, $\f$ preserves irreducible group components. Hence, by Lemma~\ref{ML}, $s$ is non-adjacent to at most one vertex in a different component, $s'=\t(\f(s))$ being the only candidate. If $s'$ is indeed in a different component, then $\t$ interchanges the component containing $s$ with that which contains $s'$. In this case, $s$ is adjacent to every vertex in its irreducible group component. It follows that the only reducible Coxeter system with disconnected $G(w_0,\t)$ is $I_2(2)$ with $\t$ interchanging the two generators.

From now on, let us restrict attention to finite, irreducible Coxeter systems. It is known, see, e.g., \cite[Sections~3.19 and 6.3]{Humphreys}, that $\f$ is not the identity involution if and only if the system has an even exponent, i.e., if and only if it is of type~$A_n$ ($n \geq 2$), $D_{2m+1}$, $E_6$, or $I_2(2m+1)$, for integral $m$.

First, suppose $\f=\t$. Then Lemma~\ref{ML} shows that $G(w_0,\t)$ is the complement of the Coxeter graph. That is, $s$ and $s'$ are connected if and only if they commute. The finite, irreducible Coxeter systems with disconnected complement of the Coxeter graph are those of rank $2$ and $3$ and that of type~$D_4$. With the requirement $\f=\t$ they comprise the following list: $A_3$ ($\t \neq \id$), $B_3$, $D_4$ ($\t=\id$), $H_3$, $I_2(2m)$ ($\t=\id$), and $I_2(2m+1)$ ($\t \neq \id$). In all cases, the complement of the Coxeter graph has two components.

Second, assume $\f \neq \t$. Only type~$D_4$ admits distinct non-trivial, involutive automorphisms; in this type, however, $\f$ is trivial. Thus, exactly one of $\f$ and $\t$ must be the identity involution. Let $\psi \in \{\f,\t\}$ denote the non-trivial involution. By Lemma~\ref{ML}, $s$ and $s'$ are adjacent in $G(w_0,\t)$ unless either $\psi(s)=s'$, or $s$ and $s'$ are both fixed by $\psi$ and $m(s,s') \geq 3$. It follows that $G(w_0,\t)$ is disconnected if and only if $|S|=2$. This accounts for the remaining dihedral cases $I_2(2m)$ ($\t \neq \id$) and $I_2(2m+1)$ ($\t=\id$), and concludes the proof.
\end{proof}

Recall that an initial move is the replacement of one element in $\Rh_\t(v)$ (for some $v \in \J(\t)$) by another, in the beginning of a reduced $\Ss$-expression for some $w \in \J(\t)$. Let a \emph{list initial move} be an initial move in which $v$ is the longest element of a $\t$-stable parabolic subgroup of one of the types listed in Lemma~\ref{disconnected}.

Having established all the necessary preliminaries, we are now in position to prove the main result.

\begin{proof}[Proof of Theorem~\ref{MT}]
Fix $w \in \J(\t)$. The result is trivially true if $w$ is the identity element. In order to induct on the rank, assume the result holds for all twisted involutions of rank less than $k=\rho(w)$. Consider first two reduced $\Ss$-expressions for $w$, $\epsilon$ and $\epsilon'$, that end with the same letter. By the induction hypothesis, they are related by a sequence of braid moves and list initial moves of rank at most $k-1$ that never interfere with the last letter; let $\epsilon \move{ind.} \epsilon'$ indicate this property.

If $s$ and $s'$ are connected by an edge in $G(w,\t)$, then there are two reduced $\Ss$-expressions $\epsilon$ and $\epsilon'$ for $w$ which are related by a braid move and end with $\s$ and $\s'$, respectively; let $\epsilon \move{br.} \epsilon'$ indicate this relationship.

Now, choose two arbitrary reduced $\Ss$-expressions for $w$, $\epsilon=[\s_1 \cd \s_k]$ and $\epsilon'=[\s_1' \cd \s_k']$. If there is a path $s_k=z_0 \to z_1 \to \cd \to z_t=s_k'$ in $G(w,\t)$, we have reduced $\Ss$-expressions for $w$ related in the following way:
\begin{align*}
\epsilon &\move{ind.} [u_0\underbrace{\cd \z_1\z_0}_{m(z_0,z_1)}] \move{br.} [u_0\underbrace{\cd \z_0\z_1}_{m(z_0,z_1)}] \move{ind.} [u_1\underbrace{\cd \z_2\z_1}_{m(z_1,z_2)}] \move{br.} [u_1\underbrace{\cd \z_1\z_2}_{m(z_1,z_2)}] \move{ind.} \cd \\
         &\move{ind.} [u_{t-1}\underbrace{\cd \z_t\z_{t-1}}_{m(z_{t-1},z_t)}] \move{br.} [u_{t-1}\underbrace{\cd \z_{t-1}\z_t}_{m(z_{t-1},z_t)}] \move{ind.} \epsilon',
\end{align*}
where the $u_i$ are reduced $\Ss$-expressions. Hence, $\epsilon$ and $\epsilon'$ are related by a sequence of braid moves, and list initial moves of rank at most $k-1$. On the other hand, if there is no such path connecting $s_k$ and $s_k'$, then it follows from Lemmas~\ref{parabolic} and \ref{disconnected} that $\epsilon$ and $\epsilon'$ are related by a list initial move of length $k$.
\end{proof}

\section{Necessary list initial moves} \label{disc}

In a reduced $\Ss$-expression, any operation that trades a prefix representing $w_0(J)$ for another is among those listed in Theorem~\ref{MT}, if $W_J$ is one of the specified parabolic subgroups. However, it is far from necessary to allow all of these list initial moves if the only objective is to connect all reduced $\Ss$-expressions that represent the same twisted involution. In fact, it follows from the proof of Theorem~\ref{MT} that it is necessary and sufficient to allow the replacement of one fixed prefix whose last letter is in one connected component of $G(w_0(J),\t)$, whenever this graph is disconnected, by one whose last letter is in the (only) other connected component. We next present one possible list of such replacements.

\begin{Thm} \label{minimal}
Let $(W,S)$ be a Coxeter system with an involutive automorphism $\t$. Suppose $J \seq S$ is $\t$-stable. Consider the following moves, with generator indexing as in Figure~\ref{indices}:
\begin{itemize}
  \item When $W_J$ is of type~$A_3$ and $\t_J \neq \id$:
\[
[\s_2\s_3\s_1\s_2] \; \longleftrightarrow \; [\s_2\s_3\s_2\s_1]
\]
  \item When $W_J$ is of type~$B_3$:
\[
[\s_1\s_2\s_3\s_1\s_2\s_1] \;\longleftrightarrow \; [\s_1\s_2\s_3\s_2\s_1\s_2]
\]
  \item When $W_J$ is of type~$D_4$ and $\t_J=\id$:
\[
[\s_4\s_2\s_1\s_3\s_2\s_1\s_3\s_4] \; \longleftrightarrow \; [\s_4\s_2\s_1\s_3\s_2\s_1\s_4\s_3]
\]
  \item When $W_J$ is of type~$H_3$:
\[
[\s_1\s_3\s_2\s_1\s_3\s_2\s_1\s_3\s_2] \; \longleftrightarrow \; [\s_1\s_3\s_2\s_1\s_3\s_2\s_1\s_2\s_3]
\]
\item When $W_J$ is of type~$I_2(m)$, $3 \leq m<\1$, and $\t_J=\id$:
\[
\underbrace{[\s_1\s_2\s_1 \cd]}_{\h{$\lceil{(m(s_1,s_2)+1)/2}\rceil$ letters}} \longleftrightarrow \underbrace{[\s_2\s_1\s_2 \cd]}_{\h{$\lceil{(m(s_1,s_2)+1)/2}\rceil$ letters}}
\]
\item When $W_J$ is of type~$I_2(m)$, $2 \leq m<\1$, and $\t_J \neq \id$:
\[
\underbrace{[\s_1\s_2\s_1 \cd]}_{\h{$\lceil{m(s_1,s_2)/2}\rceil$ letters}} \longleftrightarrow \underbrace{[\s_2\s_1\s_2 \cd]}_{\h{$\lceil{m(s_1,s_2)/2}\rceil$ letters}}
\]
\end{itemize}
If $w \in \J(\t)$, then any two reduced $\Ss$-expressions for $w$ can be connected by a sequence of braid moves and initial moves of the listed kinds.
\end{Thm}

\begin{figure}[tb]
\begin{subfigure}[b]{0.3\textwidth}
\centering
\begin{tikzpicture}[scale=1.5]
  \vertex (0) at (0,0) [label=below:$s_1$] {}; \vertex (1) at (1,0) [label=below:$s_2$] {}; \vertex (2) at (2,0) [label=below:$s_3$] {};
  \draw (0)--(1)--(2);
\end{tikzpicture}
\caption*{$A_3$}
\bigskip
\end{subfigure}
\begin{subfigure}[b]{0.3\textwidth}
\centering
\begin{tikzpicture}[scale=1.5]
  \vertex (0) at (0,0) [label=below:$s_1$] {}; \vertex (1) at (1,0) [label=below:$s_2$] {}; \vertex (2) at (2,0) [label=below:$s_3$] {};
  \draw (0)--node[above]{\footnotesize $4$}(1)--(2);
\end{tikzpicture}
\caption*{$B_3$}
\bigskip
\end{subfigure}
\begin{subfigure}[b]{0.3\textwidth}
\centering
\begin{tikzpicture}[scale=1.5]
  \vertex (x) at (1,1) [label=above:$s_1$] {};
  \vertex (0) at (0,0) [label=below:$s_2$] {}; \vertex (1) at (1,0) [label=below:$s_3$] {}; \vertex (2) at (2,0) [label=below:$s_4$] {};
  \draw (x)--(1); \draw (0)--(1)--(2);
\end{tikzpicture}
\caption*{$D_4$}
\bigskip
\end{subfigure}
\begin{subfigure}[b]{0.3\textwidth}
\centering
\begin{tikzpicture}[scale=1.5]
  \vertex (0) at (0,0) [label=below:$s_1$] {}; \vertex (1) at (1,0) [label=below:$s_2$] {}; \vertex (2) at (2,0) [label=below:$s_3$] {};
  \draw (0)--node[above]{\footnotesize $5$}(1)--(2);
\end{tikzpicture}
\caption*{$H_3$}
\bigskip
\end{subfigure}
\begin{subfigure}[b]{0.3\textwidth}
\centering
\begin{tikzpicture}[scale=1.5]
  \vertex (0) at (0,0) [label=below:$s_1$] {}; \vertex (1) at (1,0) [label=below:$s_2$] {};
  \draw (0)--node[above]{\footnotesize $m$}(1);
\end{tikzpicture}
\caption*{$I_2(m)$}
\bigskip
\end{subfigure}
\caption{The Coxeter graphs in types~$A_3$, $B_3$, $D_4$, $H_3$, and $I_2(m)$.} \label{indices}
\end{figure}
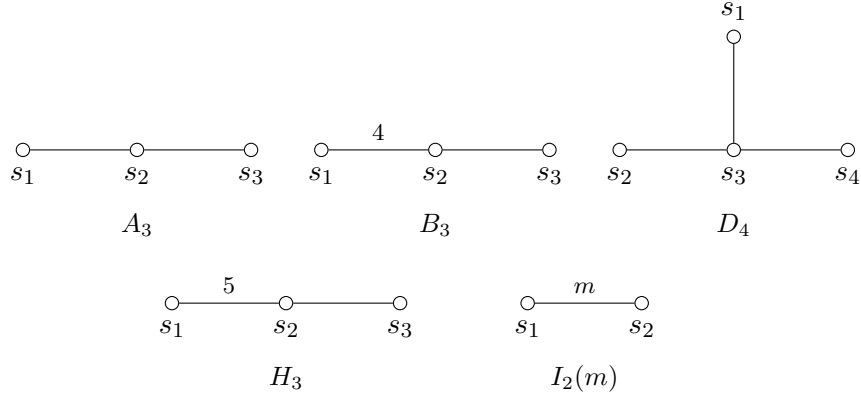

Of course, many other choices than those stated in \Th~\ref{minimal} are possible. A different selection can be found in Marberg~\cite{Marberg2}. Those we have chosen are involution braid relations in the sense of \cite{H-M-P}; recall the discussion about those from the introduction. Thus, the theorem conveys a minimal set of involution braid relations to add to the ordinary braid relations in order to connect all reduced $\Ss$-expressions of any twisted involution. As an example, Figure~\ref{big} illustrates how \Th~\ref{minimal} connects the various reduced $\Ss$-expressions for the longest element in type~$A_3$.

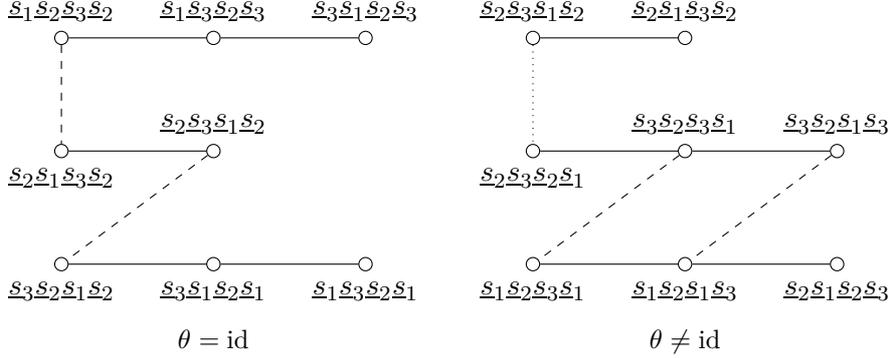
\begin{figure}[tb]
\begin{subfigure}[b]{0.48\textwidth}
\centering
\begin{tikzpicture}[xscale=2,yscale=1.5]
  \vertex (1232) at (0,2) [label=above:$\s_1\s_2\s_3\s_2$] {}; \vertex (1323) at (1,2) [label=above:$\s_1\s_3\s_2\s_3$] {};
  \vertex (3123) at (2,2) [label=above:$\s_3\s_1\s_2\s_3$] {}; \vertex (2132) at (0,1) [label=below:$\s_2\s_1\s_3\s_2$] {};
  \vertex (2312) at (1,1) [label=above:$\s_2\s_3\s_1\s_2$] {}; \vertex (3212) at (0,0) [label=below:$\s_3\s_2\s_1\s_2$] {};
  \vertex (3121) at (1,0) [label=below:$\s_3\s_1\s_2\s_1$] {}; \vertex (1321) at (2,0) [label=below:$\s_1\s_3\s_2\s_1$] {};
  \draw (1232)--(1323)--(3123);
  \draw (2132)--(2312);
  \draw (3212)--(3121)--(1321);
  \draw[dashed] (1232)--(2132);
  \draw[dashed] (2312)--(3212);
\end{tikzpicture}
\caption*{$\t=\id$}
\bigskip
\end{subfigure}
\begin{subfigure}[b]{0.48\textwidth}
\centering
\begin{tikzpicture}[xscale=2,yscale=1.5]
  \vertex (2312) at (0,2) [label=above:$\s_2\s_3\s_1\s_2$] {}; \vertex (2132) at (1,2) [label=above:$\s_2\s_1\s_3\s_2$] {};
  \vertex (2321) at (0,1) [label=below:$\s_2\s_3\s_2\s_1$] {}; \vertex (3231) at (1,1) [label=above:$\s_3\s_2\s_3\s_1$] {};
  \vertex (3213) at (2,1) [label=above:$\s_3\s_2\s_1\s_3$] {}; \vertex (1231) at (0,0) [label=below:$\s_1\s_2\s_3\s_1$] {};
  \vertex (1213) at (1,0) [label=below:$\s_1\s_2\s_1\s_3$] {}; \vertex (2123) at (2,0) [label=below:$\s_2\s_1\s_2\s_3$] {};
  \draw (2312)--(2132);
  \draw (2321)--(3231)--(3213);
  \draw (1231)--(1213)--(2123);
  \draw[dotted] (2312)--(2321);
  \draw[dashed] (3231)--(1231);
  \draw[dashed] (3213)--(1213);
\end{tikzpicture}
\caption*{$\t \neq \id$}
\bigskip
\end{subfigure}
\caption{The reduced $\Ss$-expressions for the longest element in type~$A_3$, where solid lines denote braid moves. For $\t=\id$, dashed lines indicate the $I_2(3)$ list initial move. When $\t \neq \id$, dashed lines indicate the $I_2(2)$ list initial move, and the dotted line represents the $A_3$ initial move specified in \Th~\ref{minimal}.} \label{big}
\end{figure}

\section{Special cases} \label{applications}

There are many situations where there are few list initial moves possible. In this final section, we present some consequences of the main result that arise in such settings. A \emph{half-braid move} is a list initial move of type~$I_2(m)$, see the list given in \Th~\ref{minimal}.

\begin{Cor} \label{MC}
Let $(W,S)$ be a Coxeter system with an involutive automorphism $\t$, and let $w \in \J(\t)$. Suppose $(W,S)$ does not have a $\t$-stable parabolic subgroup $W_J$ of type~$B_3$, $D_4$, or $H_3$ with $\t_J=\id$, nor one of type~$A_3$ with $\t_J \neq \id$. Then any two reduced $\Ss$-expressions for $w$ can be connected by a sequence of braid moves and half-braid moves.
\end{Cor}

Marberg~\cite[\Con~1.7]{Marberg2} conjectured that the conclusion of \Co~\ref{MC} holds whenever $\t$ fixes no element of $S$. Since the hypotheses are satisfied in that situation, we have confirmed the conjecture. Marberg's \Cons~1.8 and 1.9 also follow directly from \Th~\ref{MT}.\footnote{\cite[\Con~1.9]{Marberg2} predicts that every $\Rh_\t(w)$ can be connected using moves that satisfy certain assumptions. These assumptions imply that all braid moves and all list initial moves can be performed.}

Another interesting consequence of \Th~\ref{MT} concerns \emph{right-angled} Coxeter systems, i.e., those that satisfy $m(s,s') \in \{2,\1\}$ for all generators $s \neq s'$. Observe that if $\t=\id$, no list initial moves are available in a right-angled group. Let $I(W)$ denote the set of involutions in $W$, and recall that $\J(\id)=I(W)$.

\begin{Prop} \label{right}
If $W$ is right-angled and $\t=\id$, then the map $[s_1 \cd s_k] \mapsto [\s_1 \cd \s_k]$ sends reduced words to reduced $\Ss$-expressions, and it induces a bijection $W \to I(W)$.
\end{Prop}

\begin{proof}
In order to obtain a contradiction, assume that $[s_1 \cd s_k]$ is reduced and $[\s_1 \cd \s_k]$ is not. If $k$ is minimal among all expressions with this property, $s_k \in D_R(\s_1 \cd \s_{k-1})$. Hence, by \Th~\ref{MT}, $[\s_1 \cd \s_{k-1}]$ is related to a reduced $\Ss$-expression ending with $\s_k$ by a sequence of braid moves. But then the same sequence of braid moves transforms $[s_1 \cd s_{k-1}]$ into a word ending with $s_k$. This contradicts the reducedness of $[s_1 \cd s_k]$, proving the first claim. It is then clear that $[s_1 \cd s_k] \mapsto [\s_1 \cd \s_k]$ provides a bijection between reduced words and reduced $\Ss$-expressions. Since it respects braid moves, the second claim follows.
\end{proof}

Given a subset $X \seq W$, let $\Br(X)$ denote the poset on $X$ with the order induced by the Bruhat order on $W$.

\begin{Cor} \label{iso}
If $(W,S)$ is right-angled, then $\Br(W)$ and $\Br(I(W))$ are isomorphic as posets.
\end{Cor}

\begin{proof}
This follows from \Pro~\ref{right} together with Lemmas~\ref{subword1} and \ref{subword2}.
\end{proof}

The only finite, right-angled $(W,S)$ are of type~$A_1 \times \cd \times A_1$. In these groups, $I(W)=W$, so \Co~\ref{iso} is not particularly amusing. If $W$ is infinite, however, the inclusion $I(W) \subset W$ is proper. Hence a copy of $\Br(W)$ sits inside $\Br(W)$ as a proper subposet. Thus $\Br(W)$ contains an infinite sequence of induced subposets $P_i$, all of them isomorphic to $\Br(W)$, such that
\[
\Br(W)=P_0 \supset P_1 \supset P_2 \supset \cd.
\]

\bibliographystyle{amsplain}

\bibliography{Referenser_bara_initialer}

\end{document}